\documentclass[aip,cha,reprint,superscriptaddress,floatfix]{revtex4-2}

\usepackage[T1]{fontenc}
\usepackage[utf8]{inputenc}
\usepackage{lmodern}
\usepackage{microtype}
\usepackage{amsmath,amssymb,amsthm,mathtools}
\usepackage{bm}
\usepackage{graphicx}
\usepackage{booktabs}
\usepackage{xcolor}
\usepackage{hyperref}
\usepackage{placeins}

\hypersetup{colorlinks=true,linkcolor=blue,citecolor=blue,urlcolor=blue}
\providecommand{\Description}[1]{}

\newtheorem{theorem}{Theorem}
\newtheorem{proposition}[theorem]{Proposition}
\newtheorem{corollary}[theorem]{Corollary}
\theoremstyle{definition}
\newtheorem{definition}[theorem]{Definition}
\theoremstyle{remark}

\newcommand{\Neff}{N_{\mathrm{eff}}}
\newcommand{\aeff}{a_{\mathrm{eff}}}
\newcommand{\rhostar}{\rho^{\star}}
\newcommand{\cstar}{c^{\star}}
\newcommand{\cO}{\mathcal O}
\newcommand{\cC}{\mathcal C}
\newcommand{\cB}{\mathcal B}
\newcommand{\Rplus}{\mathbb R_{\ge 0}}

\begin{document}

\title{Scale Partitioning by Incremental Nested Entropy:\ A Measure-Oriented Theory of Multiscale Structure}

\author{Abd AlRahman R. AlMomani}
\affiliation{Embry-Riddle Aeronautical University, Prescott, Arizona, USA}
\date{\today}

\begin{abstract}
Across complex systems science, networks, spatial structures, populations, spectra, and dynamical flows are often studied through object-level magnitudes such as connectivity, influence, frequency, geometric isolation, spectral strength, or deformation. Determining whether these values contain distinct scales usually requires a chosen cutoff, a prescribed number of groups, or an assumed prevalence. We introduce Scale Partitioning by Incremental Nested Entropy (SPINE), a deterministic framework that discovers scale structure without those inputs. The first central theorem shows that the complete and potentially heterogeneous pattern of preceding values can be replaced by two interpretable descriptors: the effective number of contributing components and their characteristic scale. This reduction yields an exact criterion for when the next value creates a new scale. When many components contribute, the constant \(e\) emerges as the universal limiting balance between increasing diversity and increasing dominance, not as a fitted or user-selected threshold. A second theorem proves that the same criterion determines exactly how much relative change each detected boundary can withstand when its two sides are moved toward one another. 
By following successive transitions through the ordered data, SPINE discovers a data-supported number of scale levels and returns one unresolved scale when the measure supports no separation.
Independently located numerical transitions agree with the theory to numerical precision across markedly different value patterns, while the stability prediction remains accurate at later heterogeneous boundaries where neighboring-gap approximations fail badly. Controlled experiments recover two, three, and four generated scales without being supplied their number once the scales are sufficiently separated. In a double-gyre flow, SPINE identifies a high-expansion structure occupying \(12.85\%\) of the domain without prescribing a retained percentile. An accompanying public software implementation and cross-platform reproducibility specification translate the theory into a transparent workflow for geometric, categorical, network, spectral, and dynamical data.
\end{abstract}

\maketitle

\section{Introduction}
\label{sec:introduction}

\textbf{Many complex systems display observable properties across distinct scales, yet identifying where one scale gives way to another often requires a prescribed threshold, a chosen number of groups, or additional assumptions about the underlying distribution. Here we show that such transitions can instead be identified from how information accumulates as an ordered measure is progressively enlarged. Building on an accumulated-prefix entropy construction introduced in our earlier work, we develop a general multiscale theory in which a candidate measure-scale boundary occurs when the nested entropy changes from nondecreasing to decreasing across a strict measure gap. This transition is governed by the entropy-effective state of the complete preceding prefix, rather than by an arbitrary adjacent-value threshold. The result is an exact critical relation from which the constant \(e\) emerges in the large-effective-size limit; the same critical surface determines the aligned perturbation margin of a detected boundary. The resulting framework, Scale Partitioning by Incremental Nested Entropy (SPINE), provides a deterministic way to expose multiple entropy-resolvable scales in scientific measures without prescribing their number, while keeping the physical meaning of those scales tied to the measure chosen for the problem.}

Scientific observations frequently provide a measure of magnitude, influence, deformation, frequency, energy, distance, or importance for each object in a finite system. The difficult question is often not how to rank those values, but whether the ordered measure values contain distinct scales and, if so, where one scale ends and another begins. A threshold answers this question only after a cut has been supplied or estimated. Many clustering formulations introduce a representation model, a distance, or a prescribed number of groups. An anomaly detector may instead construct a score whose interpretation depends on a decision rule or an assumed prevalence. The SPINE framework addresses a different problem: given a scientifically meaningful nonnegative measure, determine whether its ordered values contain local entropy-supported transitions between multiplicative scales, locate those transitions, and quantify their separation. A declared orientation is then used to interpret the resulting strata. The theory is local, deterministic, and derived from Shannon entropy rather than from a parameterized probability model. Throughout this work, a \emph{scale boundary} means such an entropy-supported transition in the declared measure.

The word \emph{scale} is therefore used in a specific measure-oriented sense. Suppose that a collection of objects has already been mapped to a nonnegative scalar quantity \(m_i\). The scientific meaning of \(m_i\) is supplied by the application: it may be a neighborhood distance, a singular value, a node degree, a finite-time expansion factor, a frequency, or another nonnegative measure. SPINE does not infer this measure from raw data. It asks whether progressively larger ordered values alter the normalized allocation sufficiently to produce entropy-supported transitions in the nested Shannon profile. The resulting strata are scales of the declared measure. Their interpretation as typical, atypical, coherent, dominant, rare, or otherwise scientifically distinguished is an application-level statement. Meaningful measure construction remains a scientific modeling decision, while the entropy transition is the mathematical problem considered here.

The accumulated-prefix entropy construction underlying SPINE was first introduced in our 2020 study of frequency-ranked symptom states \cite{almomani2020symptoms}. That earlier construction ordered observed state frequencies as \(F_1\leq\cdots\leq F_{n_u}\), repeatedly normalized the first \(i\) frequencies, and evaluated the Shannon entropy of each accumulated prefix. The maximum of that sequential entropy curve was then used to define a frequency cutoff separating states associated with typical and atypical patient sets. That paper explicitly left a deeper connection to asymptotic typicality and the optimality of the cutoff for future work. The present study returns to that construction, removes its application-specific dependence on occurrence frequency, and develops the missing local scale theory: the exact one-step transition criterion, its entropy-effective state variables, the finite-size appearance of \(e\), multiscale candidate boundaries, and an exact aligned perturbation margin.

Shannon entropy itself is classical, as is its grouping identity and its interpretation as uncertainty of a normalized allocation \cite{shannon1948,coverthomas2006}. Exponentiating Shannon entropy produces an effective number of equally weighted states, an interpretation that appears naturally in Hill numbers and related diversity measures \cite{hill1973,jost2006}. The contribution of the present work is therefore not the Shannon functional, the grouping recurrence, or the effective-number transform individually. It is the general theory obtained when the ascending accumulated-prefix construction \cite{almomani2020symptoms} is expressed in entropy-effective coordinates and its entropy-supported transitions are treated as measure-scale events. The resulting nested entropy profile does not estimate differential entropy and is distinct from cumulative entropy functionals based on a distribution or survival function \cite{rao2004cumulative,dicrescenzo2009cumulative}.

There are also neighboring methods for ordered scalar hierarchy and normalized-prefix entropy that address different objectives. The head/tail breaks method recursively partitions heavy-tailed scalar data about the mean and determines a hierarchy through repeated head subdivision \cite{jiang2013headtail}; Jenks natural breaks instead optimizes within-class homogeneity for a prescribed class count \cite{jenks1967}. More recently, Top-H decoding introduced an entropy-constrained procedure for language-model sampling in which token probabilities are sorted in descending order and progressively normalized over selected prefixes \cite{potraghloou2025toph}. Although this construction also evaluates entropy over probability prefixes, it addresses a distinct problem and differs in ordering direction, scientific object, stopping criterion, and purpose. In contrast, SPINE uses ascending accumulated-prefix entropy to characterize local measure-scale transitions and derives the associated entropy-effective transition criterion and its finite-size behavior.

This measure-oriented viewpoint connects naturally to our earlier work on entropy for continuous data. Geometric partition entropy introduced a geometry-aware coarse graining for continuous state spaces \cite{diggans2022}; Boltzmann-Shannon interaction entropy combined frequency and geometric information into a normalized continuous-variable quantity \cite{diggans2023}; and generalized geometric partition entropy extended this measure-oriented construction to multidimensional entropy and mutual-information estimation in the presence of informative outliers \cite{diggans2025}. These methods construct information quantities from the measures associated with data-adaptive partitions of an observed state space. In particular, generalized geometric partition entropy provides one direct connection to the present framework: its partition defines a collection of coarse states \(A_i\), each carrying a nonnegative measure \(\mu(A_i)\). These coarse states may themselves be treated as the objects supplied to SPINE, with \(m_i=\mu(A_i)\), so that SPINE asks whether the distribution of measure across the states contains entropy-supported multiplicative scales. In this sense, the two constructions operate at successive levels: geometric partition methods provide one principled way to construct measure-bearing scientific objects from continuous data, while SPINE examines the intrinsic scale structure of the resulting measures. More generally, SPINE is not restricted to geometric partitions and can operate on any collection of scientific objects for which an appropriate nonnegative measure can be defined.

A related methodological perspective appears in information-theoretic approaches to structure discovery, where conditional information is used to determine whether candidate components contribute independently to an observed system. Entropic regression applies this principle to the identification of governing dynamical terms \cite{almomani2020}, while optimal causation entropy has been used to recover interaction structure and Boolean functions from data \cite{sun2022}. SPINE addresses a different inference problem: no governing equation, interaction network, or causal structure is sought. Instead, information is evaluated along the ordered accumulation of a prescribed nonnegative measure to determine whether the measure itself contains resolvable scale structure. The commonality is therefore methodological rather than algorithmic: information relations are used to identify structure that is not determined by magnitude alone.

The SPINE framework is distinct from established thresholding, clustering, outlier detection, and change-point formulations. Classical thresholding selects a cut to divide a scalar field or histogram \cite{otsu1979}; entropy itself has also been used to choose histogram thresholds by optimizing class entropies \cite{kapur1985}. Clustering addresses grouping in a data space and spans formulations with prescribed group counts as well as density-based methods whose number of returned clusters is data dependent \cite{jain2010,ester1996}. Robust outlier procedures identify observations inconsistent with a bulk model or robust location and scale \cite{tukey1977,hampel1974}; density and isolation methods construct other notions of local abnormality \cite{breunig2000,liu2008,chandola2009}; and information-theoretic atypicality provides another route to data discovery and anomaly detection through coding ideas \cite{hostmadsen2019}. Change-point methods seek changes along an externally meaningful sequence through a segment cost and search formulation \cite{killick2012,truong2020}. SPINE does not replace these formulations. Its input is an already defined scalar measure, and its object of inference is a transition in the multiplicative scale of that ordered measure. A returned stratum may later be interpreted as atypical, coherent, dominant, rare, or otherwise scientifically distinguished, but that semantic interpretation is separate from the scale partition itself.

The theoretical difficulty is that the entropy change caused by admitting each new measure value appears, at first, to depend on the complete composition of the preceding prefix. Importantly, two prefixes can have the same arithmetic mean yet distribute their total measure very differently, and their normalized entropies can therefore respond differently to the same incoming value. A useful scale theory requires a state reduction that preserves this one-step entropy response without retaining the full prefix vector. 


The central result of SPINE is precisely such a reduction. For a prefix with total measure \(S\) and Shannon entropy \(E\), define the entropy-effective size \(\Neff=2^E\) and the entropy-effective scale \(\aeff=S/\Neff\). The interpretation of \(\Neff\) follows the standard interpretation of exponentiated Shannon entropy as the effective number of equally weighted states \cite{hill1973,jost2006}. The direction of the next entropy increment depends on the complete prefix only through \(\Neff\) and the incoming measure relative to \(\aeff\). This yields a unique finite-size critical ratio \(\rhostar(\Neff)\). The resulting transition criterion has three consequences that organize the paper.

First, \(\rhostar>e\) for every finite positive-entropy prefix, and \(\rhostar\to e\) as the effective size grows. The constant \(e\) therefore enters as a derived asymptotic balance point, not as a heuristic separation threshold. Second, transitions from nondecreasing to decreasing entropy across strict measure gaps define candidate boundaries that compose into an unknown number of measure strata without prescribing that number. Third, the same entropy-effective state governs robustness: the excess of an incoming scale above its critical value determines an exact relative perturbation radius under aligned erosion of the scale gap. Boundary formation and aligned stability are therefore two consequences of the same local geometry.

The numerical experiments are used to expose these theoretical properties rather than to select algorithmic rules. They show that entropy-effective normalization collapses heterogeneous prefixes onto a single finite-size relation, that the same local state predicts the stability of first and later multiscale boundaries, and that successive local transitions can recover multiple generated scale levels without being supplied their number. A finite-time deformation example then converts the abstract critical ratio into a physically interpretable condition on accumulated logarithmic separation. Extended proofs, the optional global selector, sensitivity studies, additional applications, further numerical results, and boundary cases are provided in the Supplementary Material.

The paper proceeds as follows. Section~\ref{sec:nested} defines nested entropy for an ordered measure, derives the exact recurrence, and introduces the entropy-effective state. Section~\ref{sec:critical} derives the entropy-effective transition criterion, relates arithmetic and entropy-effective normalization through the entropy deficit, establishes the finite-size behavior and appearance of \(e\), and identifies when the critical transition is compatible with the ascending order constraint. Section~\ref{sec:partition} turns entropy-supported transitions in the nested profile into multiscale partitions. Section~\ref{sec:stability} derives the exact aligned boundary-stability radius and proves that, for an actual candidate under aligned erosion, the strict entropy-decrease condition fails before the adjacent measure gap closes. Section~\ref{sec:results} presents the focused numerical evidence and the dynamical application, followed by the broader interpretation and limitations in Sec.~\ref{sec:discussion}.

\section{Nested entropy of an ordered measure and the entropy-effective state}
\label{sec:nested}

Let \(\cO_N=\{o_1,\ldots,o_N\}\) be a finite set of objects, and let \(\mu:\cO_N\to\Rplus\) be a declared scientific measure, with \(m_i=\mu(o_i)\). The map \(\mu\) assigns a nonnegative weight to each object; it is not a probability distribution. The measure may carry physical units and may arise from geometry, dynamics, counts, spectral decomposition, network structure, or another application-specific construction. Probability enters only after normalization of a prefix. Let \(m_{(1)}\le\cdots\le m_{(N)}\) denote a stable nondecreasing ordering of the measure values, with object identities carried through the ordering, and let \(S_n=\sum_{j=1}^n m_{(j)}\) be the cumulative measure of the first \(n\) ordered objects. 

Whenever \(S_n>0\), the first \(n\) ordered values induce normalized weights \(p_j^{(n)}=m_{(j)}/S_n\). These weights sum to one and therefore define a finite probability vector on the prefix, but their meaning is purely measure-normalized. Unless \(\mu\) is itself an occurrence frequency, \(p_j^{(n)}\) should not be read as the probability that object \(o_{(j)}\) occurs. Rather, it is the fraction of accumulated prefix measure carried by that object. This distinction is useful in applications because the same entropy construction can act on measures with very different scientific meanings.

\begin{definition}[Incremental nested entropy]
\label{def:nested}
The nested entropy profile is the sequence \(E_1,\ldots,E_N\) obtained by evaluating Shannon entropy on each normalized ordered prefix,
\begin{equation}
E_n=-\sum_{j=1}^{n}p_j^{(n)}\log_2 p_j^{(n)}.
\label{eq:nested-entropy}
\end{equation}
If \(S_n=0\), we set \(E_n=0\). Zero-mass terms contribute zero by continuity. For \(n\ge2\), the local entropy increment is \(\Delta_n=E_n-E_{n-1}\).
\end{definition}

The profile in Eq.~\eqref{eq:nested-entropy} is a sequence of entropies of different normalized allocations, not an entropy of rank and not a histogram entropy of the original sample. Ordering determines which object enters next, while renormalization determines how the accumulated measure is redistributed after that entry. If all \(n\) positive values in a prefix are equal, then the normalized allocation is uniform and \(E_n=\log_2 n\). An exactly homogeneous measure scale therefore produces a strictly increasing profile. At the opposite extreme, an incoming value that captures nearly all of the enlarged prefix measure drives the normalized allocation toward a single dominant component and can reduce entropy sharply. The sign of each local entropy increment is determined by the competition between these two tendencies.

This competition is explicit in the one-step recurrence, which is the Shannon grouping identity specialized to the ascending nested-prefix construction \cite{shannon1948,coverthomas2006}. Let \(q_n=m_{(n)}/S_n\) be the normalized mass of the incoming value after it has entered the prefix, and let \(h_2(q)=-q\log_2q-(1-q)\log_2(1-q)\) denote binary entropy.

\begin{theorem}[One-step entropy recurrence]
\label{thm:recurrence}
For \(n\ge2\) with \(S_n>0\),
\begin{equation}
\begin{aligned}
E_n
&=(1-q_n)E_{n-1}+h_2(q_n),\\
\Delta_n
&=h_2(q_n)-q_nE_{n-1}.
\end{aligned}
\label{eq:recurrence}
\end{equation}
\end{theorem}
\begin{proof}
If \(S_{n-1}=0\), then \(q_n=1\) and both identities hold with \(E_{n-1}=E_n=h_2(1)=0\). Otherwise, after the new value enters, each of the first \(n-1\) normalized masses is multiplied by the common factor \(1-q_n\), while the new component has normalized mass \(q_n\). Expanding the Shannon entropy of the vector \(((1-q_n)\bm p^{(n-1)},q_n)\) gives the first identity. Subtracting \(E_{n-1}\) gives the second. A closed-form identity equivalent to the recurrence is given in Supplementary Note~S1.1.
\end{proof}

Equation~\eqref{eq:recurrence} isolates the scale-transition mechanism without introducing any fitted parameter. The binary entropy term \(h_2(q_n)\) is the diversification contribution from admitting another component. The contraction term \(q_nE_{n-1}\) is the entropy lost because all previously accumulated normalized masses are compressed by the new component. When \(q_n\) is small, diversification dominates. When \(q_n\) is sufficiently large, concentration dominates. The sign of \(\Delta_n\) records the balance.

The recurrence also clarifies the zero-entropy boundary case. If a positive prefix has \(E_{n-1}=0\), its normalized measure is concentrated on one positive component. Admitting any finite additional positive value creates at least two positive normalized masses, so the new entropy is strictly positive. Consequently, an entropy decline cannot occur from a zero-entropy prefix. The nontrivial transition criterion developed below therefore applies only when the effective size exceeds one.

\begin{proposition}[Positive scale invariance]
\label{prop:scale-invariance}
For any \(\alpha>0\), replacing every measure value by \(\alpha m_i\) leaves the nested entropy profile, all candidate boundary ranks, and the induced scale partition unchanged.
\end{proposition}

\begin{proof}
Positive rescaling preserves the order and multiplies every prefix sum by \(\alpha\). Hence each normalized weight \(p_j^{(n)}\) is unchanged, and so is every entropy in Eq.~\eqref{eq:nested-entropy}. The candidate conditions introduced in Sec.~\ref{sec:partition} depend only on these entropies and on strict order gaps, which are also preserved by positive rescaling.
\end{proof}

Positive scale invariance separates the theory from the units used to express \(\mu\). A distance measured in meters or centimeters, a singular value multiplied by a common calibration factor, or a frequency vector expressed in proportional units produces the same partition. SPINE concerns relative measure scale, not a universal calibration of raw magnitude.

\subsection{The entropy-effective state}
\label{sec:effective-state}

The recurrence shows that the next entropy increment depends on the previous entropy and on the normalized mass of the incoming value. For scale interpretation, however, a dimensionless fraction \(q\) alone is not enough: the same fraction can arise from different total prefix measures. The critical reduction becomes transparent after combining the entropy and total measure into two quantities with complementary meanings.

\begin{definition}[Entropy-effective state]
\label{def:effective-state}
For a prefix with positive total measure, its entropy-effective size and entropy-effective scale are
\begin{equation}
\Neff=2^{E_n},
\qquad
\aeff=\frac{S_n}{\Neff}.
\label{eq:effective-state}
\end{equation}
\end{definition}

The effective size \(\Neff\) is the number of equally weighted states having the same Shannon entropy as the actual normalized prefix, consistent with the equivalent-number interpretation of entropy \cite{hill1973,jost2006}. Thus \(1\le\Neff\le n\), with \(\Neff=n\) only for a uniform positive prefix. The effective scale \(\aeff\) has the same physical units as the original measure. It is the common measure that \(\Neff\) equal effective components would carry if they preserved the same total prefix measure \(S_n\). In practical terms, \(\Neff\) describes how many effective components are present, while \(\aeff\) describes the characteristic measure carried by one such component.

The effective scale also has an intrinsic representation that does not refer to a hypothetical
uniform system.

\begin{proposition}[Measure-weighted geometric representation]
\label{prop:geometric-mean}
Let \(J_n=\{j\le n:m_{(j)}>0\}\) denote the positive support of the prefix; then we have
\begin{equation}
\aeff
=
\prod_{j\in J_n}m_{(j)}^{p_j^{(n)}}.
\label{eq:geometric-mean}
\end{equation}
\end{proposition}

\begin{proof}
If the measure carries physical units, take all logarithms relative to one arbitrary positive reference unit; the reference cancels because the normalized weights sum to one. Set \(L_n=\sum_{j\in J_n}p_j^{(n)}\log_2m_{(j)}\) in those common units. Expanding Eq.~\eqref{eq:nested-entropy} gives \(E_n=\log_2S_n-L_n\). Hence \(\log_2\aeff=\log_2S_n-E_n=L_n\), and exponentiation gives Eq.~\eqref{eq:geometric-mean}.
\end{proof}

Equation~\eqref{eq:geometric-mean} is useful conceptually because it shows that \(\aeff\) is not an arbitrary rescaling introduced to simplify the critical relation. It is a measure-weighted geometric mean of the positive prefix values, with the same normalized weights that define the entropy. Large measure values receive proportionally larger weights, while the logarithmic averaging preserves multiplicative scale. This makes \(\aeff\) a natural reference when the scientific distinction of interest is multiplicative rather than additive.

The difference between raw and effective prefix size can be quantified by the entropy deficit \(\delta_n=\log_2n-E_n\ge0\). Since \(\Neff=2^{E_n}\), one has \(n/\Neff=2^{\delta_n}\). A uniform prefix has zero deficit and \(\Neff=n\); heterogeneous prefixes have positive deficit and behave, for the next entropy transition, as smaller effective systems. This distinction becomes essential in Sec.~\ref{sec:critical}, where the complete dependence on prefix composition is reduced to \(\Neff\) and \(\aeff\).

\section{The entropy-effective critical transition}
\label{sec:critical}

Consider a positive ordered prefix ending at rank \(n\), with effective state defined by Eq.~\eqref{eq:effective-state}, and a positive incoming value \(c\), without yet imposing the ordering constraint \(c\ge m_{(n)}\). Define the dimensionless incoming ratio \(\rho=c/\aeff\). Since \(S_n=\Neff\aeff\), the normalized mass carried by the incoming value after admission is \(q=\rho/(\Neff+\rho)\). The central task is to determine the value of \(\rho\) for which the entropy of the enlarged prefix is exactly equal to \(E_n\). Ordered admissibility of this formal equality is considered in Sec.~\ref{sec:admissibility}.

\begin{theorem}[Entropy-effective transition criterion]
\label{thm:critical}
Let \(x=\Neff>1\). There exists a unique critical ratio \(\rhostar(x)>1\) satisfying
\begin{equation}
\rhostar\ln\rhostar
=
\bigl(x+\rhostar\bigr)
\ln\!\left(1+\frac{\rhostar}{x}\right).
\label{eq:critical-law}
\end{equation}
For an incoming value \(c=\rho\aeff\), entropy increases when \(\rho<\rhostar(x)\), is unchanged at equality, and decreases when \(\rho>\rhostar(x)\). The detailed composition of the prefix enters this one-step decision only through \(x=2^{E_n}\).
\end{theorem}

\begin{proof}
Apply the recurrence in Eq.~\eqref{eq:recurrence} to the enlarged prefix. Substituting \(q=\rho/(x+\rho)\) and \(E_n=\log_2x\) reduces the entropy increment to
\[
E_{n+1}-E_n
=
-\frac{F_x(\rho)}{(x+\rho)\ln2},
\]
where \(F_x(\rho)=\rho\ln\rho-(x+\rho)\ln(1+\rho/x)\) and \(\ln\) denotes the natural logarithm. Hence the entropy-neutral condition is \(F_x(\rho)=0\), which is Eq.~\eqref{eq:critical-law}. The derivative and second derivative are \(\partial F_x/\partial\rho=\ln[x\rho/(x+\rho)]\) and \(\partial^2F_x/\partial\rho^2=x/[\rho(x+\rho)]>0\). Thus \(F_x\) is strictly convex and has a single stationary point, at \(\rho=x/(x-1)\). Moreover, \(F_x(1)<0\), while \(F_x(\rho)\to\infty\) as \(\rho\to\infty\). The increasing branch therefore crosses zero exactly once for \(\rho>1\). The sign of \(E_{n+1}-E_n\) is the opposite of the sign of \(F_x\), giving the three stated regimes. The complete algebraic reduction is given in Supplementary Note~S1.3.
\end{proof}

Theorem~\ref{thm:critical} is the central state reduction. For the purpose of deciding the direction of the next entropy increment, a heterogeneous prefix is transition-equivalent to an effective system of \(x\) equal components, each of measure \(\aeff\). Two prefixes that have different raw sizes, different arithmetic means, and different internal distributions produce the same one-step transition whenever they have the same effective size and are challenged by the same incoming ratio \(\rho\). The theory therefore replaces a high-dimensional prefix description by an effective state consisting of a dimensionless diversity coordinate \(x\) and a dimensional scale coordinate \(\aeff\).

The relevant comparison is not generally the adjacent raw ratio \(m_{(n+1)}/m_{(n)}\). The terminal value \(m_{(n)}\) records only one point of the prefix, whereas \(\aeff\) summarizes the entire normalized allocation through Eq.~\eqref{eq:geometric-mean}. In a homogeneous prefix, \(m_{(n)}=\aeff\), but after earlier scales have entered, these quantities can differ substantially. This distinction is responsible for both the universal critical collapse and the multiscale stability result below.

The role of prefix heterogeneity can be stated exactly in terms of the entropy deficit introduced in Sec.~\ref{sec:effective-state}. Let \(\overline m_n=S_n/n\) denote the arithmetic mean measure of the prefix, and let \(\cstar=\rhostar(\Neff)\aeff\) be the formal entropy-neutral incoming value.

\begin{corollary}[Arithmetic and entropy-effective normalization]
\label{cor:arithmetic-normalization}
Let \(\overline m_n=S_n/n\) denote the arithmetic mean of the prefix,
let \(\cstar=\rhostar(\Neff)\aeff\) denote the formal entropy-neutral
incoming value, and let \(\delta_n\) denote the entropy deficit, \(\delta_n=\log_2n-E_n\ge0\). 
Then the critical ratio relative to the arithmetic mean is:
\begin{equation}
\frac{\cstar}{\overline m_n}
=
2^{\delta_n}\rhostar(\Neff),
\qquad
\delta_n=\log_2n-E_n.
\label{eq:arithmetic-link}
\end{equation}
\end{corollary}

\begin{proof}
Using \(\aeff=S_n/\Neff\), one obtains \(\cstar/\overline m_n=(n/\Neff)\rhostar\). Since \(n/\Neff=2^{\delta_n}\), Eq.~\eqref{eq:arithmetic-link} follows.
\end{proof}

Corollary~\ref{cor:arithmetic-normalization} separates two sources of variation that are otherwise mixed together. The critical-ratio function itself depends only on \(\Neff\). Arithmetic normalization introduces the entropy-deficit factor \(2^{\delta_n}\), so families with different internal heterogeneity need not share a common arithmetic threshold. Entropy-effective normalization removes this factor by construction. This is the theoretical reason for the numerical collapse in Sec.~\ref{sec:result-critical}.

\subsection{Finite-size behavior and the appearance of \texorpdfstring{\(e\)}{e}}
\label{sec:e}

The critical ratio is a finite-size function, not a fixed constant. Its large-effective-size limit nevertheless has a simple exact form and a direct interpretation.

\begin{theorem}[Finite-size behavior of the critical ratio]
\label{thm:finite-size}
For every finite \(x>1\), \(\rhostar(x)>e\), and \(\rhostar(x)\) decreases monotonically to \(e\) as \(x\to\infty\). As \(x\to\infty\), its first finite-size correction is
\begin{equation}
\rhostar(x)=e+\frac{e^2}{2x}+O(x^{-2}).
\label{eq:first-correction}
\end{equation}
\end{theorem}

\begin{proof}
At \(\rho=e\), write \(y=e/x>0\). Then \(F_x(e)=x[y-(1+y)\ln(1+y)]<0\), because \((1+y)\ln(1+y)-y\) is zero at \(y=0\) and has strictly positive derivative for \(y>0\). The unique root therefore satisfies \(\rhostar(x)>e\). At the root, \(\partial F_x/\partial\rho>0\), while \(\partial F_x/\partial x=\rho/x-\ln(1+\rho/x)>0\). Implicit differentiation gives \(d\rhostar/dx<0\). The root is decreasing and bounded below by \(e\), so it has a finite limit \(L\ge e\). Since the root remains bounded, the right side of Eq.~\eqref{eq:critical-law} tends to \(L\), giving \(L\ln L=L\). The solution with \(L>1\) is \(L=e\). Expanding the implicit equation about \((x^{-1},\rho)=(0,e)\) gives Eq.~\eqref{eq:first-correction}. The analytic higher-order expansion through \(O(x^{-5})\) is derived in Supplementary Note~S1.4.
\end{proof}

The finite-size correction is not merely a formal asymptotic detail. A prefix with small effective size can require a critical ratio far above \(e\). For example, Eq.~\eqref{eq:critical-law} gives \(\rhostar(2)\approx6.807\), while \(\rhostar(10)\approx3.135\). The exact identity \(\rhostar(4)=4\) provides a useful finite-size anchor: substituting \(\rho=x\) into Eq.~\eqref{eq:critical-law} reduces the equality to \(x\ln x=2x\ln2\), whose positive solution is \(x=4\). The numerical consequences of this finite-size behavior are shown in Fig.~\ref{fig:critical}.

The appearance of \(e\) has a specific meaning. Adding another component tends to increase entropy because the support of the normalized allocation becomes richer. At the same time, the incoming component forces all previous normalized masses to contract. If the incoming measure is large enough relative to the effective scale of the prefix, concentration dominates and entropy decreases. The critical ratio is the unique balance point between these effects. As \(x\to\infty\), the right side of Eq.~\eqref{eq:critical-law} approaches \(\rho\), so the balance reduces to \(\rho\ln\rho=\rho\). For \(\rho>1\), this gives \(\ln\rho=1\) and hence \(\rho=e\).

This interpretation should not be taken to mean that \(e\) is a universal scale-separation threshold. The constant is universal only within the asymptotic form of this nested Shannon transition. Finite prefixes are governed by \(\rhostar(x)>e\), and the incoming measure is compared with \(\aeff\), not generally with the immediately preceding value. Consequently, an adjacent raw ratio smaller than \(e\) does not by itself rule out a transition, and an adjacent raw ratio larger than \(e\) does not by itself establish one. The relevant quantity is the complete entropy-effective state of the prefix.

\subsection{Ordered admissibility of the formal equality}
\label{sec:admissibility}

As noted in the derivation of Theorem~\ref{thm:critical}, the transition criterion was obtained for an arbitrary positive incoming value \(c\), without imposing the ordering constraint \(c\ge m_{(n)}\) that SPINE actually encounters. The formal entropy-neutral value \(\cstar=\rhostar(\Neff)\aeff\) is therefore compatible with the ordered continuation only if it is no smaller than the current terminal value.

\begin{proposition}[Ordered admissibility]
\label{prop:ordered-admissibility}
If \(\cstar\ge m_{(n)}\), the formal entropy equality lies within the admissible ordered continuation. If \(\cstar<m_{(n)}\), every admissible next value \(c\ge m_{(n)}\) lies on the entropy-decreasing side of Theorem~\ref{thm:critical}.
\end{proposition}


The second case, \(\cstar<m_{(n)}\), does not invalidate the transition criterion. It means that the ordered sequence has already moved beyond the formal equality point, so every admissible continuation lies on the entropy-decreasing side. 
This cannot occur for a homogeneous prefix, where \(\aeff=m_{(n)}\) and hence \(\cstar=\rhostar(\Neff)\,m_{(n)}>m_{(n)}\), but it
is common in heterogeneous prefixes that already contain strongly elevated values; the detailed admissibility study is given in Supplementary Notes~S1.7 and S3.1. 
The distinction is useful because the critical relation describes an entropy balance, while the ordering constraint determines whether
that balance is still reachable by the next observed measure value.

\section{From entropy-supported transitions to multiscale partitions}
\label{sec:partition}

The transition criterion determines the sign of a single local entropy increment. SPINE uses successive transitions from nondecreasing to decreasing entropy to partition the ordered measure. The direct construction, denoted SPINE-C, retains every entropy-supported candidate. An optional selected mode, SPINE-S, can subsequently choose a global subset of the candidate set according to a declared penalized representation objective, but the scale-transition theory itself does not depend on that selector.

A candidate boundary must satisfy two distinct conditions. First, the nested entropy must stop rising and begin to fall. Second, the ordered measure must actually increase across the proposed break. The first condition identifies a local transition from nondecreasing to decreasing entropy; the second prevents an arbitrary split inside a block of equal measure values.

\begin{definition}[Candidate scale boundary]
\label{def:candidate}
For \(N\ge3\), the SPINE-C candidate set is
\begin{equation}
\begin{split}
\cC=\{i\in\{2,\ldots,N-1\}:{}&E_i\ge E_{i-1},\\
&E_i>E_{i+1},\ m_{(i)}<m_{(i+1)}\}.
\end{split}
\label{eq:candidate}
\end{equation}
For \(N<3\), \(\cC\) is empty.
\end{definition}

The asymmetry in the entropy comparisons is deliberate. The non-strict condition on the left admits the final point of an exact entropy plateau before a decline, while the strict condition on the right requires that a decline actually occur after rank \(i\). The strict measure gap places the boundary between distinct represented measure values. Hence a candidate at rank \(i\) has a precise local meaning: adding \(m_{(i)}\) does not lower the prefix entropy, but adding \(m_{(i+1)}\) does, and the two ordered values belong to different measure levels rather than to the same tie block. No candidate is defined at the first or final rank because both neighboring entropy comparisons are required.

The transition criterion explains this rule directly. At a candidate \(i\), adding \(m_{(i)}\) does not decrease entropy, whereas the next value has an entropy-effective ratio above \(\rhostar(2^{E_i})\) relative to the prefix ending at \(i\). The boundary therefore marks an observed entry into the entropy-decreasing side of the local balance. This is why the candidate set is defined from the entropy profile rather than from a raw gap statistic.

Let the retained boundaries be \(\cB=\{b_1<\cdots<b_K\}\subseteq\cC\), with \(b_0=0\) and \(b_{K+1}=N\). These boundaries define \(G=K+1\) consecutive scale strata, with stratum \(k\) containing ordered ranks from \(b_{k-1}+1\) through \(b_k\). If \(K=0\), the complete object system is one unresolved measure scale. In SPINE-C, \(\cB=\cC\), so the number of strata is determined entirely by the entropy-supported candidate set. No number of groups, contamination fraction, or target prevalence is supplied.

The output is a partition of objects, not merely a list of threshold values. Because object identities are carried through the stable ordering, every stratum maps back to the original object system. This distinction matters when several scale transitions are present. A single cut can separate only two sides of a scalar ordering, whereas a sequence of entropy-supported boundaries can represent several distinct measure levels without converting the problem into repeated binary thresholding.

Relative separation coordinates can also be assigned to the resulting strata. Each retained boundary has a positive entropy-drop depth obtained from the decline following its local maximum, and cumulative normalized drop depth maps the ordered strata from separation level zero to one. These coordinates summarize how strongly successive scale transitions contribute to the total discovered separation. The declared orientation is then applied to interpret the ordered strata scientifically. Upper orientation treats larger measure scales as increasingly atypical relative to the typical stratum; lower orientation treats smaller measure scales as increasingly atypical; bilateral orientation measures departure from a declared anchor stratum in both directions. The typical stratum has score exactly zero. These scores are deterministic coordinates on the discovered scale structure, not probabilities, posterior anomaly scores, or estimates of occurrence likelihood. Exact scoring definitions are given in Supplementary Note~S2.3.

Orientation is therefore a statement about the scientific meaning of the measure, not a numerical switch that changes the partitioning theory. For a nearest-neighbor distance, large values may encode geometric isolation and upper orientation may be natural. For frequency, small values may encode rarity and lower orientation may be appropriate. For a spectral or deformation measure, large values may represent dominant modes or strong expansion. The same partition can have different scientific meanings under different orientations, so the orientation should be declared based on the application rather than inferred from the entropy profile.

SPINE-S is an optional representation layer for applications that require a global subset of the entropy-supported candidates. It restricts all possible boundaries to \(\cC\), fits each candidate-delimited segment by weighted isotonic regression, and minimizes a declared penalized representation objective. The selector uses classical weighted isotonic regression \cite{barlow1972,robertson1988,best1990}, together with a penalty motivated by coding and model-selection ideas \cite{rissanen1978,schwarz1978}. This objective is separate from the Shannon transition theorem. Its penalty is neither likelihood-derived nor assumed to vary monotonically with model complexity. Because the search is restricted to entropy-supported candidates, SPINE-S cannot introduce a boundary where no SPINE-C candidate exists. Its exact objective, dynamic programming algorithm, deterministic tie-breaking hierarchy, and penalty behavior are given in Supplementary Note~S2.4. All main numerical results below use SPINE-C unless stated otherwise.

\section{Stability of an entropy-supported boundary}
\label{sec:stability}

A detected boundary is more informative if its margin can be quantified. The critical surface provides such a margin without introducing a new statistical model or a new scale parameter. Let \(b\in\cC\) be a candidate boundary. Define the effective state of the prefix ending at \(b\) and the entropy-effective ratio of the incoming upper value by
\begin{equation}
x_b=2^{E_b},
\qquad
a_b=\frac{S_b}{x_b},
\qquad
\rho_b=\frac{m_{(b+1)}}{a_b}.
\label{eq:boundary-state}
\end{equation}
Because the candidate condition includes \(E_{b+1}<E_b\), Theorem~\ref{thm:critical} implies \(\rho_b>\rhostar(x_b)\). The amount by which \(\rho_b\) exceeds the critical ratio measures the local separation from criticality and, under the aligned perturbation defined below, determines how much erosion the boundary can absorb before the entropy decrease disappears.

We consider first an aligned relative perturbation that directly erodes the boundary. Every measure at ranks up to \(b\) is multiplied by \(1+\eta\), while every measure above \(b\) is multiplied by \(1-\eta\), with \(0\le\eta<1\). The complete lower prefix is therefore rescaled uniformly. Its normalized distribution, entropy, and effective size are unchanged; its effective scale is multiplied by \(1+\eta\), while the incoming upper value is multiplied by \(1-\eta\). This perturbation is deterministic and adversarial with respect to the local separation because it moves the two sides directly toward one another.

An exact stability statement for the same object partition also requires the ordered gap between \(m_{(b)}\) and \(m_{(b+1)}\) to remain open. For a generic cut, an order collision could occur before an entropy condition fails. A SPINE candidate has additional structure that rules this out.

\begin{proposition}[Candidate order margin]
\label{prop:order-margin}
For a candidate boundary \(b\),
\begin{equation}
\frac{m_{(b)}}{a_b}<\rhostar(x_b).
\label{eq:terminal-prefix-bound}
\end{equation}
Consequently, under aligned erosion, the strict right-hand entropy condition fails before the adjacent measure gap can close.
\end{proposition}

\begin{proof}
Let \(q_b=m_{(b)}/S_b\) and define \(g(q)=h_2(q)/q\) for \(q\in(0,1)\). The left candidate condition \(E_b-E_{b-1}\ge0\), together with Eq.~\eqref{eq:recurrence}, gives \(g(q_b)\ge E_{b-1}\). Using the same recurrence to eliminate \(E_{b-1}\) yields \(g(q_b)\ge E_b\). Let \(q_b^\star\) be the unique fraction satisfying \(g(q_b^\star)=E_b=\log_2x_b\). Since \(g\) is strictly decreasing, \(q_b\le q_b^\star\). The critical ratio corresponding to the prefix ending at \(b\) satisfies \(\rhostar(x_b)=x_bq_b^\star/(1-q_b^\star)\), while \(m_{(b)}/a_b=q_bx_b\). Hence \(m_{(b)}/a_b\le q_b^\star x_b<\rhostar(x_b)\), proving Eq.~\eqref{eq:terminal-prefix-bound}. The comparison of the entropy and order radii then follows by direct cross multiplication; the complete argument is given in Supplementary Note~S1.8.
\end{proof}

The proposition is more than a technical ordering check. It shows that a candidate carries a built-in asymmetry: the last value inside the lower prefix remains below the critical scale of that prefix, while the first value above the boundary lies above it. The entropy-critical value \(\rhostar(x_b)a_b\) therefore lies strictly inside the raw measure gap associated with an actual candidate.

\begin{theorem}[Exact aligned boundary-stability radius]
\label{thm:stability}
For a candidate boundary \(b\), the split between the same two object sets remains a candidate boundary under the aligned perturbation exactly for \(0\le\eta<\eta_b^\star\), where
\begin{equation}
\eta_b^\star
=
\frac{\rho_b-\rhostar(x_b)}{\rho_b+\rhostar(x_b)}.
\label{eq:stability-radius}
\end{equation}
At \(\eta=\eta_b^\star\), the strict entropy decline at the right side of the boundary becomes entropy-neutral while the adjacent measure gap remains open.
\end{theorem}

\begin{proof}
Uniform rescaling leaves \(x_b\) unchanged and maps \(a_b\) to \((1+\eta)a_b\). The incoming value maps to \((1-\eta)m_{(b+1)}\), so its effective ratio becomes \(\rho_b(1-\eta)/(1+\eta)\). The strict right-hand entropy decline persists exactly while this ratio exceeds \(\rhostar(x_b)\). Solving equality gives Eq.~\eqref{eq:stability-radius}. Every prefix ending at or below \(b\) is uniformly rescaled, so the left candidate comparison remains unchanged. Proposition~\ref{prop:order-margin} guarantees that the adjacent measure gap remains open until after the entropy equality is reached. Thus Eq.~\eqref{eq:stability-radius} is the exact preservation radius of the same object split under the aligned perturbation.
\end{proof}

The radius has a transparent interpretation. If \(\rho_b\) lies only slightly above \(\rhostar(x_b)\), the boundary is near criticality and \(\eta_b^\star\) is small. If \(\rho_b\) is much larger than the critical ratio, the boundary tolerates correspondingly larger relative erosion. The radius therefore converts the binary statement that a boundary exists into a quantitative local margin.

For a homogeneous lower prefix of \(b\) equal positive measures and an upper-to-lower raw ratio \(r=m_{(b+1)}/m_{(b)}\), one has \(x_b=b\), \(a_b=m_{(b)}\), and \(\rho_b=r\). Theorem~\ref{thm:stability} then reduces to \(\eta^\star=[r-\rhostar(b)]/[r+\rhostar(b)]\). This is the special case in which the adjacent raw ratio and the entropy-effective ratio coincide. At later boundaries in a heterogeneous multiscale prefix, \(a_b\) generally differs from \(m_{(b)}\), and the entropy-effective ratio is the correct quantity.

Theorem~\ref{thm:stability} unifies boundary formation and aligned robustness. The critical surface \(\rho=\rhostar(x)\) determines when an incoming scale first produces a descending entropy increment. Once the boundary exists, the same surface determines how much aligned erosion it can tolerate before the entropy decline disappears. No independent robustness parameter is introduced. Detection and aligned stability are therefore two manifestations of one entropy-effective law.

More general perturbations need not rescale the lower prefix uniformly, so they can change both \(E_b\) and \(a_b\) in ways not captured by a single aligned ratio. A distribution-free sufficient guarantee can then be obtained by combining strict candidate margins with finite-alphabet entropy continuity bounds \cite{fannes1973,audenaert2007}. Such certificates answer a broader but different question: they guarantee candidate preservation for all perturbations satisfying a specified discrepancy bound, rather than locating the exact failure point for the aligned geometry. We place that general theorem and its numerical tightness analysis in Supplementary Notes~S1.9 and S3.3. The supplementary results also show that much of the conservatism of a generic certificate can come from bounding the induced total variation too crudely rather than from the entropy-continuity theorem itself.

\section{Numerical manifestations and a dynamical application}
\label{sec:results}


The experiments below expose consequences of the theory; none is used to tune the candidate definition, the critical ratio, or the stability formula. Where a predicted critical value is tested, the measured value is obtained by solving the entropy-neutral condition \(E_{n+1}=E_n\) numerically, so it is independent of the formula being tested. The main text retains only results bearing on the
critical relation, multiscale stability, recovery of an unknown number of scales, and one physical consequence; extended experiments, sensitivity studies, negative cases, and additional applications are reported in Supplementary Note~S3.

\subsection{Finite-size critical relation and entropy-effective collapse}
\label{sec:result-critical}

Figure~\ref{fig:critical} illustrates two consequences of Theorems~\ref{thm:critical} and \ref{thm:finite-size}: the finite-size approach to \(e\), and the result that \(\aeff\) removes the internal composition factor identified in Corollary~\ref{cor:arithmetic-normalization}. The exact critical root was evaluated over 900 logarithmically spaced effective sizes from \(1.05\) to \(10^7\). It remained above \(e\) throughout, with representative values \(\rhostar(2)=6.8069957581\), \(\rhostar(4)=4\), \(\rhostar(10)=3.1346674313\), and \(\rhostar(2000)=2.7201301392\). The first-order asymptotic approximation is already accurate for moderate effective sizes, while the higher-order expansion in Supplementary Note~S1.4 extends that accuracy to substantially smaller effective sizes.

\begin{figure*}[t]
\centering
\includegraphics[width=0.96\textwidth]{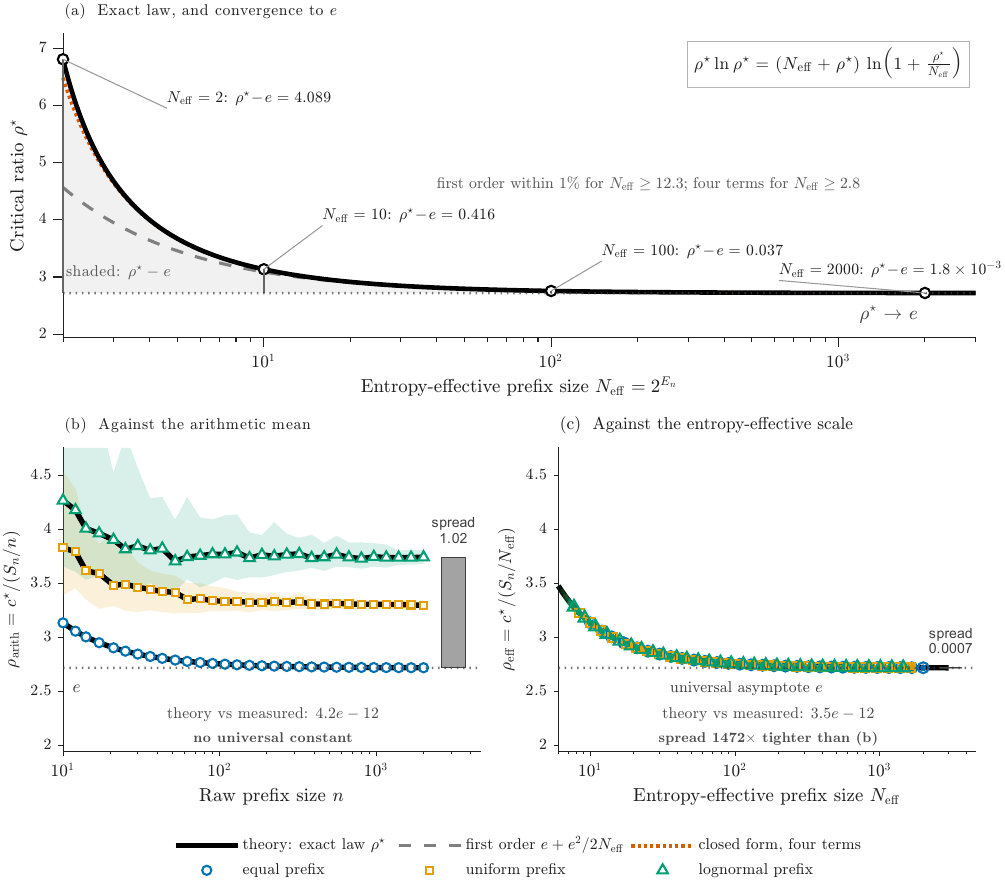}
\caption{\textbf{Finite-size critical relation and entropy-effective collapse.} (a) The exact critical ratio decreases toward \(e\) as \(\Neff\) grows; the finite-size correction is substantial for small effective sizes. (b) Critical incoming values normalized by the arithmetic prefix mean remain family dependent for equal, uniform, and lognormal prefixes. (c) The same independently located critical values normalized by \(\aeff\) collapse onto the theoretical critical-ratio function \(\rhostar(\Neff)\). Ribbons show the 10th to 90th percentiles over 80 realizations.}
\Description{Three panels show the finite-size critical ratio and the collapse obtained by entropy-effective normalization. The first panel approaches e from above. Arithmetic normalization leaves three prefix families separated, while entropy-effective normalization places them on the same curve.}
\label{fig:critical}
\end{figure*}

To test composition dependence, 7200 prefixes were generated from equal, uniform, and lognormal families over 30 raw prefix sizes and 80 realizations per family and size. For each prefix, \(\cstar\) was found by bisection from \(E_{n+1}(\cstar)=E_n\), without using Eq.~\eqref{eq:critical-law}. The measured root is therefore independent of the formula being tested. Under arithmetic normalization, the three families remain separated because Eq.~\eqref{eq:arithmetic-link} retains the entropy-deficit factor \(2^{\delta_n}\). At raw size \(n=2000\), the arithmetic-normalized critical ratios were approximately 2.7201, 3.2974, and 3.7405 for equal, uniform, and lognormal prefixes, respectively, giving a family spread of 1.020342.

The same measured values collapse when divided by \(\aeff\). At \(n=2000\), the three entropy-effective ratios were approximately 2.7201, 2.7205, and 2.7208, with spread \(6.93\times10^{-4}\). Relative to arithmetic normalization, the family spread is reduced by a factor of approximately 1472. More importantly, the individual points follow the theoretical curve at their own effective sizes: the maximum absolute discrepancy between the independently measured entropy-effective roots and \(\rhostar(\Neff)\) was \(3.485\times10^{-12}\). Additional gamma and Pareto families preserve the same behavior, as reported in Supplementary Note~S3.6.

The collapse is the central result of this experiment. It does not merely show that one normalization produces a visually tighter plot than another. Corollary~\ref{cor:arithmetic-normalization} predicts exactly why arithmetic scaling remains family dependent and why the effective scale removes that dependence. The experiment therefore supports the interpretation of \((\Neff,\aeff)\) as the state variables of the one-step transition rather than as post hoc descriptive statistics.

\subsection{The same critical surface controls multiscale stability}
\label{sec:result-stability}

Theorem~\ref{thm:stability} was first tested in the simpler two-scale setting and then at later boundaries in genuinely multiscale prefixes. In the two-scale study, the measured radii across 42 configurations spanning total sample size, upper-stratum fraction, and scale ratios from 3 to 50 agreed with the exact radius to a maximum relative error of \(1.05\times10^{-8}\). This agreement provides a numerical check of the special case in which the lower prefix is homogeneous and \(a_b=m_{(b)}\).

The more demanding test used 26 boundaries from 11 configurations with three and four scales in different arrangements. Later boundaries are preceded by heterogeneous prefixes that have already absorbed one or more lower scales. They therefore distinguish the entropy-effective reference scale from the adjacent raw value and directly test the local state reduction beyond the first boundary.

\begin{figure*}[t]
\centering
\includegraphics[width=0.87\textwidth]{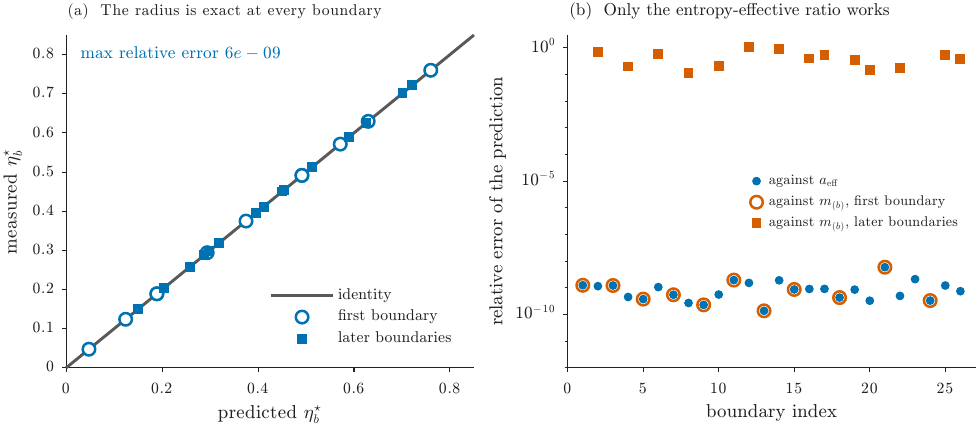}
\caption{\textbf{Local stability is governed by the entropy-effective prefix state.} (a) Measured boundary-loss radii under aligned erosion for 26 first and later boundaries agree with the prediction in Eq.~\eqref{eq:stability-radius}, with a maximum relative error of \(5.93\times10^{-9}\). (b) Replacing \(a_b\) with the adjacent lower value \(m_{(b)}\) is accurate at first homogeneous boundaries but produces relative errors of order one at later heterogeneous boundaries.}
\Description{Two panels compare measured and predicted perturbation radii. The entropy-effective predictions agree with the measured radii along the identity line. Errors remain near numerical precision for the entropy-effective formula and become large when the adjacent raw measure is used at later boundaries.}
\label{fig:stability}
\end{figure*}

The entropy-effective formula agreed with every measured radius with a maximum relative error of \(5.93\times10^{-9}\). Replacing \(a_b\) by \(m_{(b)}\) produced relative errors as large as 3.249. The adjacent-value expression succeeds at a first boundary with a homogeneous lower prefix because the terminal lower value equals the entropy-effective scale. Once earlier scales have entered the prefix, that equality disappears and the entropy-effective prefix state becomes essential.

This distinction is conceptually important. If stability were governed by the visible adjacent gap alone, the entropy-effective state reduction would be a first-boundary phenomenon. Instead, the same entropy-effective ratio that creates a later boundary also determines the amount of aligned erosion required for it to lose candidate status. The experiment therefore provides a multiscale test of the theory rather than only a numerical check of a two-level formula. Supplementary Note~S3.2 gives the complete configurations, and Supplementary Note~S3.3 compares the exact aligned radius with the broader continuity-based certificate.

\subsection{Discovering an unknown number of scales}
\label{sec:result-multiscale}

A local transition theory becomes a multiscale theory only if successive entropy-supported boundaries compose into an interpretable partition. We therefore generated measures with two, three, or four equal-size geometric levels, \(1,r,r^2,\ldots\), using \(N=480\) objects and lognormal multiplicative within-level jitter with log-scale standard deviation 0.03. One hundred independent realizations were generated for each combination of level count and adjacent ratio. A realization counted as successful only when SPINE recovered the exact number of generated scale levels, so partial recovery and overpartitioning were both counted as failures.

\begin{figure*}[t]
\centering
\includegraphics[width=0.96\textwidth]{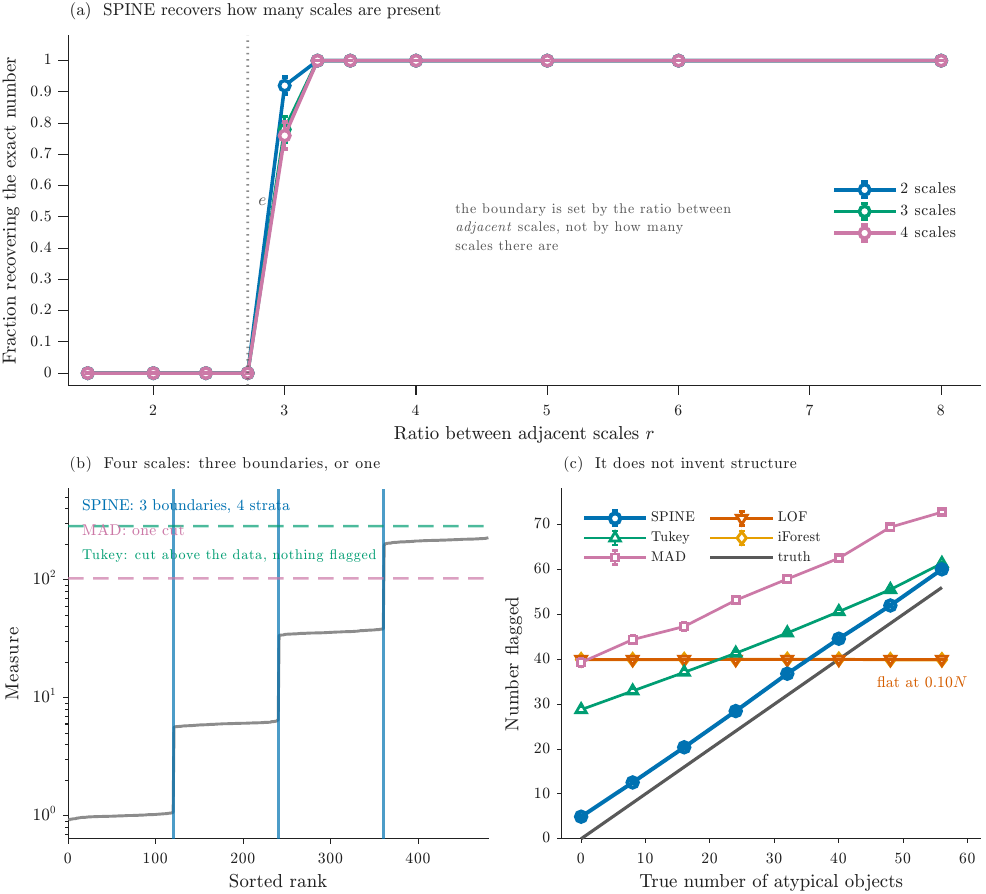}
\caption{\textbf{Scale structure is inferred rather than prescribed.} (a) Exact recovery probability for measures with two, three, and four generated geometric scale levels as the adjacent ratio \(r\) increases. The three recovery transitions nearly coincide. (b) In a representative four-level measure with \(r=6\), SPINE places three boundaries at the known level breaks. Standard single-cut rules cannot represent four strata without additional structure. (c) In a planar core-plus-annulus problem, the number assigned to the upper atypical scale changes with the specified atypical prevalence, while methods supplied with contamination 0.10 remain near 40 objects by construction.}
\Description{Three panels show exact recovery of multiple generated scales, a representative four-scale partition, and the number of atypical objects recovered as the specified prevalence changes. The scale-count recovery curves for two, three, and four levels are nearly coincident.}
\label{fig:multiscale}
\end{figure*}

Exact recovery was not observed for the tested ratios at or below 2.72, occurred in some realizations at \(r=3\), and occurred in all 100 realizations for all three level counts from \(r=3.25\) onward. At \(r=3\), the exact-recovery fractions were 0.92, 0.78, and 0.76 for two, three, and four scales, respectively. The recovery transition is therefore not identical across level counts near its onset, but the curves occupy the same narrow separation region and reach probability one at the same tested ratio.

The near coincidence is the notable feature. Within this controlled equal-size geometric family, adding separated scales does not force the required adjacent ratio to grow progressively with the total number of levels. This is consistent with a local picture in which each new boundary is resolved against the entropy-effective state of the prefix that precedes it. The experiment does not prove that arbitrary mixtures or unequal level sizes share this property, but it shows that successive local transitions can compose without an evident increase in the required separation attributable solely to the number of levels.

In the representative four-level realization in Fig.~\ref{fig:multiscale}(b), SPINE places boundaries at ranks 120, 240, and 360, exactly at the known level breaks. The comparison with single-cut robust rules \cite{tukey1977,hampel1974} is intentionally structural. A single cut can describe at most two groups, whereas the SPINE partition reports every entropy-supported transition encountered by the ordered measure. A multilevel thresholding procedure could, of course, be supplied with or optimized over several cuts, but that is a different formulation from inferring successive entropy-supported boundaries without first declaring how many levels are sought.

The value 3.25 is not a method constant. In the sensitivity sweep, the first fully successful grid point was 2.75 when the generated levels contained no within-level variation, close to the entropy-critical value, and increased as within-level spread increased. At logarithmic jitter standard deviations 0.01, 0.03, 0.06, 0.12, and 0.20, the corresponding recovery points were 3.00, 3.25, 3.50, 4.00, and 7.25. The practical resolution margin therefore reflects a competition between between-scale separation and within-scale variation. The critical relation provides the zero-spread reference, while within-level variation consumes part of the available separation margin.

Figure~\ref{fig:multiscale}(c) illustrates a related but distinct consequence. In a planar Gaussian core with a generated annular population, SPINE was supplied only with the 18th-nearest-neighbor distance as its measure under upper orientation. As the specified annular count changed from 0 to 16, 32, and 56, the mean number assigned to the upper scale changed from 4.9 to 20.4, 36.8, and 60.0. LOF \cite{breunig2000} and isolation forest \cite{liu2008} were each supplied with contamination 0.10 and therefore remained near 40 objects throughout. This is not a claim that one formulation dominates the others. The parameterized methods return the prevalence they are instructed to seek; SPINE instead returns the size of the scale supported by the supplied measure. The nonzero mean of 4.9 when no annular population is present also shows that finite-sample variation in the geometric measure can itself form a small upper scale. Additional comparisons, including cases in which the declared measure does not support a meaningful binary interpretation, are given in Supplementary Notes~S3.5 and S5.

\subsection{A dynamical consequence: a minimum accumulated separation}
\label{sec:result-lcs}

The critical ratio becomes especially interpretable when a positive scientific measure has an exponential representation over an observation horizon. We use the canonical double gyre as a controlled example and treat the largest finite-time singular-value expansion factor \(\sigma_{\max}(\bm x_0,T)\) of the flow-map gradient as the positive measure. The associated finite-time Lyapunov exponent is \(\Lambda(\bm x_0,T)=|T|^{-1}\ln\sigma_{\max}(\bm x_0,T)\). FTLE ridges have long been used as a practical diagnostic for hyperbolic finite-time coherent structures \cite{shadden2005}, while broader Lagrangian coherent-structure theory makes clear that coherent structures are not exhausted by scalar ridge thresholding \cite{haller2015}. Our use here is deliberately limited to the FTLE-based setting.

The expansion factor, rather than the FTLE itself, is the natural SPINE measure because it is positive and is exponentially related to the FTLE over the same observation horizon. Let \(T>0\) denote the magnitude of that horizon. Let the incoming expansion factor be \(c(T)=\exp[\Lambda_{\mathrm{in}}(T)T]\), and let \(a_{\mathrm{eff}}(T)\) be the entropy-effective scale of the preceding expansion-factor prefix. Because \(a_{\mathrm{eff}}(T)\) is the measure-weighted geometric mean of the positive prefix values, it can be written as \(a_{\mathrm{eff}}(T)=\exp[\Lambda_{\mathrm{eff}}(T)T]\), where \(\Lambda_{\mathrm{eff}}(T)=T^{-1}\ln a_{\mathrm{eff}}(T)\) is the entropy-effective prefix exponent. Hence the exact dimensionless incoming ratio is
\[
\rho(T)=\exp\!\left(\left[\Lambda_{\mathrm{in}}(T)-\Lambda_{\mathrm{eff}}(T)\right]T\right).
\]
Defining \(\Delta\Lambda_{\mathrm{eff}}(T)=\Lambda_{\mathrm{in}}(T)-\Lambda_{\mathrm{eff}}(T)\), Theorem~\ref{thm:critical} gives
\begin{equation}
\begin{aligned}
\Delta\Lambda_{\mathrm{eff}}(T)\,T
&>\ln\rhostar(\Neff(T)),\\
\ln\rhostar(\Neff)
&\longrightarrow 1
\quad\text{as}\quad
\Neff\to\infty.
\end{aligned}
\label{eq:dynamical-criterion}
\end{equation}
Thus an incoming deformation scale is entropy-resolvable when its accumulated logarithmic separation from the entropy-effective prefix exceeds the finite-size threshold \(\ln\rhostar(\Neff)\). In the large-effective-size regime, this threshold approaches one.

\begin{figure*}[p]
\centering
\includegraphics[width=0.92\textwidth]{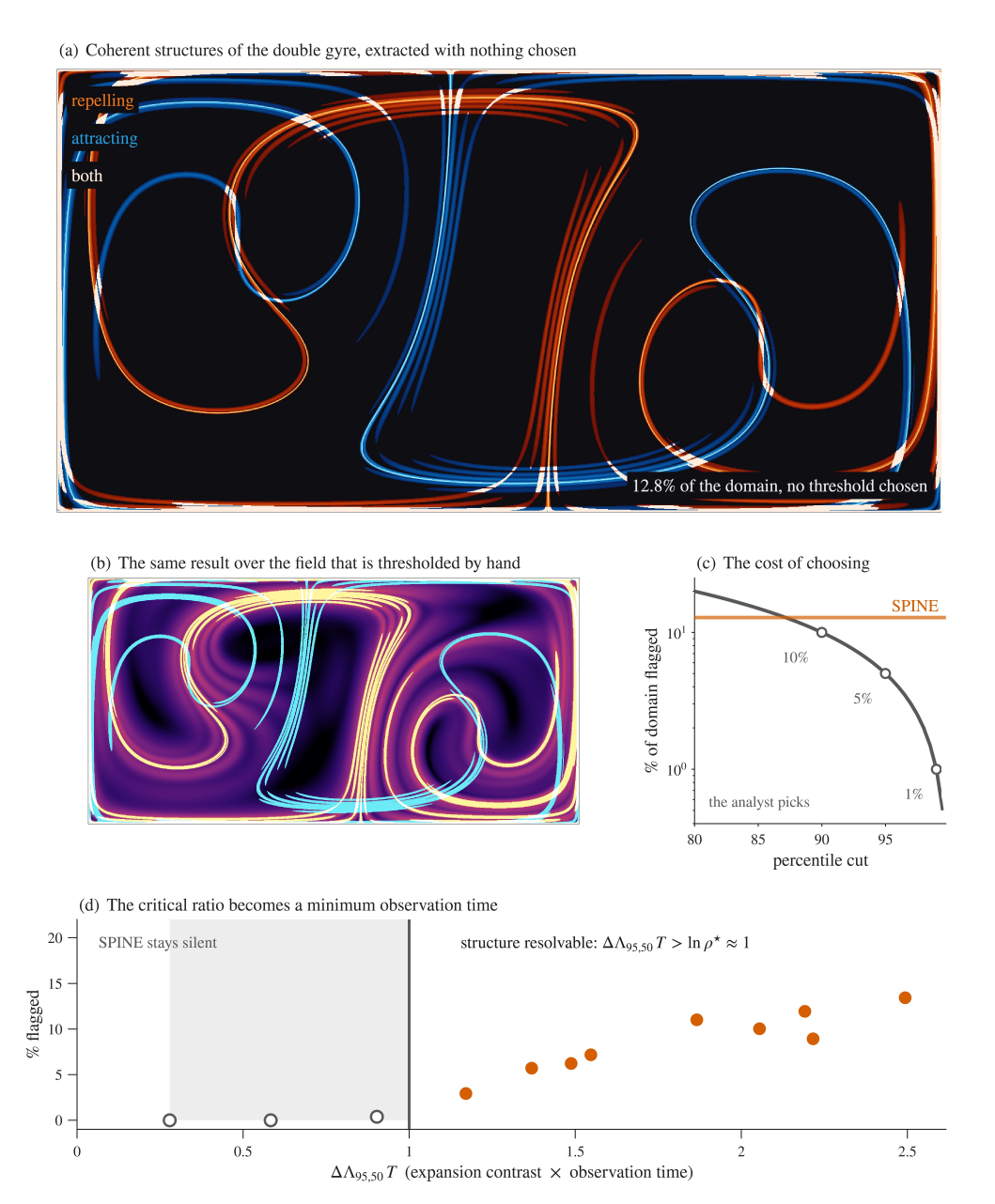}
\caption{\textbf{FTLE-based finite-time structure without a user-specified retention percentile in the double gyre.} (a) Repelling and attracting high-expansion structures inferred from the expansion-factor measures, occupying 12.85\% of the domain at the displayed horizon without prescribing a retained fraction. (b) The same inferred structures overlaid on the FTLE field. (c) Conventional percentile cuts retain exactly the fraction chosen by the analyst, while SPINE returns a measure-supported scale fraction. (d) A field-based accumulated FTLE contrast, defined independently of SPINE, enters the order-one regime as the high-expansion scale emerges. Open markers below 0.5\% retained fraction are a plotting convention, not an algorithmic threshold.}
\Description{Four panels show FTLE-based high-expansion structures in the double gyre, the structures overlaid on the FTLE field, the dependence of percentile thresholding on the chosen percentile, and the emergence of the SPINE high-expansion scale as the accumulated field contrast approaches one.}
\label{fig:lcs}
\end{figure*}

At the principal observation horizon, SPINE identifies a high-expansion stratum occupying 12.85\% of the computational domain. No retained fraction or percentile is supplied. By comparison, the 90th, 95th, and 99th percentile cuts retain 10\%, 5\%, and 1\% by construction. The point is not that one of those percentiles is intrinsically incorrect. They encode different externally selected questions. SPINE asks whether the expansion factors themselves support a distinct upper scale and returns the size of that scale as an output.

A horizon sweep provided an independent consistency check rather than a direct evaluation of Eq.~\eqref{eq:dynamical-criterion}. For each \(T\), a field-level contrast \(\Delta\Lambda_{95,50}(T)\) was defined as the 95th percentile minus the median of the FTLE field, without using the SPINE boundary. SPINE returned one stratum when the accumulated field contrast \(\Delta\Lambda_{95,50}(T)T\) was 0.279 or 0.583. A second scale first appeared in the sweep at an accumulated contrast of 0.903, containing only 0.38\% of the domain, and expanded to 2.93\% at a contrast of 1.171. It then grew to 6.23\%, 11.01\%, and 13.43\% at larger accumulated contrasts of 1.488, 1.866, and 2.493.

The exact theorem concerns \(\Delta\Lambda_{\mathrm{eff}}(T)T\), not \(\Delta\Lambda_{95,50}(T)T\). The percentile-minus-median statistic is deliberately independent of the detected boundary and is used only to ask whether the physical field enters an order-one accumulated-contrast regime when a distinct upper expansion scale begins to appear. The observed emergence in the neighborhood of one is therefore qualitatively consistent with the asymptotic transition, not a numerical verification of the exact criterion.

The result also gives a practical interpretation of the observation horizon. If a deformation contrast is weak, a short trajectory may not accumulate enough logarithmic separation for the expansion factors to form distinct scales, even if the underlying stretching rates differ. Increasing \(T\) can make the distinction resolvable, but \(\Delta\Lambda_{\mathrm{eff}}(T)\) itself changes with the horizon, so a long-time contrast cannot simply be multiplied by an arbitrary short \(T\). The relevant exact quantity is the finite-horizon product \(\Delta\Lambda_{\mathrm{eff}}(T)T\), with both factors evaluated over the same observation window. The choice of \(T\) remains a physical modeling decision; SPINE removes the additional extraction percentile, not the need to choose the dynamical observation horizon.

The extended validation is important because the example is spatially resolved. Forward and backward fields satisfy the expected symmetry of the double gyre to within \(4.201\times10^{-11}\), and the independently located SPINE cuts have a relative discrepancy of \(4.022\times10^{-14}\). Grid refinement from \(401\times201\) to \(1601\times801\) changes the retained fraction from 14.57\% to 12.85\%, while deterministic thinning of the large sorted measure changes the located cut and retained fraction by no more than 0.03\% over the tested thinning sizes. These checks, together with flow-parameter sensitivity, are reported in Supplementary Note~S3.7. The example complements our earlier work on coherent organization from image data \cite{almomani2020coherence} but does not reuse that algorithm or claim equivalence among different coherent-structure definitions.

\section{Discussion}
\label{sec:discussion}

The central contribution of SPINE is a local state reduction for measure-scale transitions. The pair \((\Neff,\aeff)\) compresses the normalized composition of an ordered measure prefix into two coordinates with complementary roles: the dimensionless effective size \(\Neff\) records how many equally weighted components would carry the same Shannon entropy, while the dimensional effective scale \(\aeff\) provides the multiplicative reference for the next incoming value. Prefixes with different raw cardinalities, arithmetic means, and internal distributions therefore have the same one-step entropy response whenever they share the same effective size and are evaluated at the same dimensionless incoming ratio. The geometric representation in Eq.~\eqref{eq:geometric-mean} shows how heterogeneity changes both the effective number of components and their characteristic measure scale. Corollary~\ref{cor:arithmetic-normalization} further shows that arithmetic normalization mixes these effects through the entropy deficit, whereas entropy-effective coordinates separate them and expose the common finite-size critical relation.

The appearance of \(e\) in this relation is exact, but asymptotic. It is derived from the entropy balance rather than fitted or imposed as an algorithmic threshold. At finite effective size, the critical ratio satisfies \(\rhostar(\Neff)>e\), and heterogeneity can make a large raw prefix respond like a much smaller effective system. Only as \(\Neff\to\infty\) does the balance reduce to \(\rho\ln\rho=\rho\), making \(e\) the point at which the diversification gained by admitting another effective component is exactly offset by concentration into that component. A fixed \(e\) cutoff applied to raw adjacent ratios would therefore discard both the finite-size dependence and the correct reference scale. This distinction becomes especially important at later boundaries, where the terminal value \(m_{(b)}\) need not represent the heterogeneous lower prefix, while \(a_b\) remains tied to its complete normalized composition. The stability experiments confirm this consequence: the adjacent-value expression is exact for a homogeneous lower prefix but can develop order-one errors at later heterogeneous boundaries.

The stability theorem reveals that boundary formation and aligned robustness are governed by the same critical surface. At a candidate boundary, the right-side entropy-decrease condition is equivalent to \(\rho_b>\rhostar(x_b)\), and the exact aligned erosion radius is determined by the distance of \(\rho_b\) above that critical value. Proposition~\ref{prop:order-margin} further establishes that the entropy-decrease condition fails before the adjacent measure gap closes under aligned erosion. A candidate only slightly beyond criticality is therefore valid according to the deterministic boundary rule but has a small perturbation margin, whereas a boundary farther beyond criticality has a larger certified margin. Boundary existence and boundary strength are thus two consequences of the same local entropy geometry. The radius is a deterministic perturbation quantity, not a probability or statistical confidence level, and it requires no separate robustness parameter or secondary calibration.

The multiscale experiment provides evidence that these local transitions can compose without an evident resolution penalty caused solely by the number of levels present. Within the tested equal-size geometric family, measures with two, three, and four generated scales enter nearly the same narrow recovery regime. This is an empirical composition result rather than a general recovery theorem. The present mathematics controls the sign of each local entropy increment and the aligned stability of a specified candidate split, but it does not establish consistent recovery for arbitrary unequal, overlapping, or nongeometrically spaced strata. The result nevertheless demonstrates that successive entropy-supported boundaries can form an interpretable multiscale partition without requiring the number of scales to be specified in advance.

The local construction also explains why SPINE can return one stratum. If the nested entropy profile produces no valid candidate boundary, the appropriate output is a single unresolved measure scale rather than a forced partition. Conversely, a scale-free measure may span a large dynamic range without supporting a preferred semantic division. Supplementary Note~S5 develops both cases. The absence of a boundary is therefore an interpretable outcome, not automatically an algorithmic failure. This feature acts as a safeguard against confusing scale partitioning with the assignment of externally desired labels.

The choice of measure is where domain science enters. SPINE does not remove scientific modeling; it separates measure construction from scale inference. A nearest-neighbor distance asks a geometric isolation question, a frequency asks a rarity question, a node degree asks a connectivity question, a singular value asks a spectral-scale question, and a finite-time expansion factor asks a deformation question. The same entropy construction can analyze each measure, but the resulting strata do not have interchangeable meanings. Supplementary Note~S3.8 illustrates this distinction across geometry, category counts, networks, and spectra. Orientation carries the same scientific responsibility. It does not change the discovered partition, but determines whether larger values, smaller values, or departures in both directions are interpreted as increasingly atypical relative to the designated typical stratum.

SPINE occupies a distinct position relative to thresholding, anomaly detection, and clustering. When only one boundary is present, its output can be expressed as a threshold, but that threshold is a consequence of an entropy-supported transition rather than the primary optimization variable. An upper or lower stratum may be interpreted scientifically as atypical, but SPINE does not begin with a probabilistic contamination model or a target prevalence. It can also produce several strata, but it does not optimize a distance-based cluster representation of the original objects. These alternatives remain valid for the questions they are designed to answer. SPINE addresses a different question: whether a declared scalar measure contains entropy-resolvable multiplicative scales.

SPINE has a direct mathematical lineage. The ascending accumulated-prefix entropy curve, together with its use in separating frequency-ranked typical and atypical states, was introduced in our 2020 symptom-informativity study \cite{almomani2020symptoms}. What was application-specific and heuristic there is made general and analytic here. The present theory identifies the entropy-effective state of any prefix with positive total measure, derives its exact finite-size critical relation and asymptotic constant, converts entropy-supported transitions into candidate measure-scale boundaries, and determines their exact aligned stability margins. 

An entropy-supported boundary is a property of the finite ordered measure, not automatically a certificate of a latent mixture component, physical regime, or statistically significant change point. A smooth one-component heavy-tailed sample can itself contain candidate transitions because extreme order statistics may become large relative to the entropy-effective state of the preceding prefix. Similarly, SPINE cannot recover a scientific distinction that is absent from the declared measure. Physical or statistical interpretation therefore requires assumptions connecting the measure and its discovered scales to the scientific system. SPINE determines the scale structure supported by the measure; it does not independently establish the semantic origin of that structure.

Additional limitations concern the geometry of practical resolution. Zero-entropy prefixes create structural edge cases: an all-zero lower stratum and a positive singleton lower stratum cannot produce the required entropy decline at their first positive transition. Within-scale spread also competes with between-scale separation, so recovery of a noisy generated level requires enough margin for internal variation not to erase the candidate transition. Under sufficiently strong correlated deformation, the observed measure may develop competing scales that no longer correspond to the originally intended structure. These behaviors, derived and tested in Supplementary Note~S5, do not alter the exact critical relation for a given ordered prefix. They delimit when an entropy-supported finite-sample partition can be identified with a particular generated or physical structure.

The finite-time application points to a broader consequence for positive measures that accumulate multiplicatively. When a measure can be represented as \(m(T)=\exp[\Lambda(T)T]\), a ratio of measure scales becomes an accumulated difference between logarithmic growth rates. Equation~\eqref{eq:dynamical-criterion} then gives a minimum accumulated separation relative to the entropy-effective prefix. The double gyre provides a controlled illustration of this relation, but the mathematical consequence is not specific to that flow. Similar reasoning may apply when amplification, decay, or growth generates the positive measure being partitioned, provided that the measure and its logarithmic rate remain scientifically meaningful. The application nevertheless concerns FTLE-based scale extraction and should not be interpreted as a universal definition of Lagrangian coherent structure.

Several theoretical extensions follow naturally from the present results. The population limit of the centered nested entropy profile, developed under integrability conditions in Supplementary Note~S1.11, opens a route from finite ordered measures to distribution-level scale geometry. The general preservation theorem based on entropy continuity supplies guarantees beyond aligned perturbations, while the certificate study shows how incorporating known perturbation structure can sharpen those guarantees. A further open problem is to determine conditions under which successive entropy-effective states characterize a global scale partition and yield convergence guarantees for candidate locations. These directions would extend the exact local theory established here into a broader statistical theory of multiscale recovery.

SPINE should therefore be viewed as a theory first and an algorithm second. The deterministic algorithm directly realizes the mathematical construction: order the declared measure, evaluate the nested entropy profile, detect valid candidate boundaries, and map them back to object strata. The public Python and MATLAB implementations make this construction reproducible and convenient, but the scientific claims do not depend on software-specific conventions. The Supplementary Material separates the mathematical specification, extended evidence, and numerical conformance requirements so that the theory can be implemented and verified independently across computational platforms.

\section{Conclusion}
\label{sec:conclusion}

SPINE provides a deterministic, measure-oriented framework for discovering entropy-supported scale structure in an ordered nonnegative measure. It identifies an un-predefined number of scale levels without requiring a user-specified cut, group count, contamination fraction, or target prevalence. The resulting boundaries define a partition of the original objects, retain multiple supported transitions, and return a single unresolved stratum when the measure does not support a separation.

The framework is supported by an exact analytical foundation. The entropy-effective state reduces a heterogeneous prefix to the two quantities that govern its next transition, yielding a common finite-size critical relation whose asymptotic limit is \(e\). The same state determines the correct reference scale for an incoming value and the exact aligned stability radius of every detected boundary. Scale detection and boundary robustness are therefore derived from one entropy construction rather than introduced through separate rules or calibrated parameters.

The numerical evidence confirms these theoretical strengths across heterogeneous prefix families, first and later multiscale boundaries, measures with an unknown number of generated levels, and a spatially resolved finite-time deformation problem. Independently reproducible Python and MATLAB implementations produce consistent results across computational platforms. SPINE therefore converts a declared scientific measure into an interpretable and reproducible scale partition, providing a general route for exposing multiscale organization in complex systems, networks, geometric data, spectra, and dynamical processes.

\section*{Supplementary Material}
The Supplementary Material contains complete derivations and extended proofs, including the higher-order critical-ratio expansion, ordered admissibility, the exact order-margin argument, general perturbation certificates, and the population profile; the mathematical specification of SPINE-C and the optional SPINE-S selector; extended numerical results for admissibility, stability, rank selection, multiscale recovery, sensitivity, finite-time deformation, and measure choice; and a concise cross-platform reproducibility specification.

\section*{Acknowledgments}

\section*{Author Declarations}

\subsection*{Conflict of Interest}
The author declares no conflicts of interest.

\section*{Data Availability}
The source code, numerical data, fixed experimental inputs, and Python and MATLAB workflows required to reproduce the results reported in this study are publicly available at \url{https://github.com/almomaa/entropy-spine}. Complete documentation is available at \url{https://almomaa.github.io/entropy-spine/}. All reported examples use computationally generated benchmark systems; no proprietary or restricted data are required.

\bibliography{SPINE_references}

\end{document}